\RequirePackage{ifpdf}
\ifpdf 
\documentclass[pdftex]{sigma}
\else
\documentclass{sigma}
\fi

\newcommand{\defn}{\ensuremath{\overset{\mathrm{def}}{=}}}
\newcommand{\dd}{\ensuremath{\mathrm{d}}}

\begin{document}

\renewcommand{\PaperNumber}{033}

\FirstPageHeading

\ShortArticleName{The Fundamental $k$-Form and Global Relations}
\ArticleName{The Fundamental $\boldsymbol{k}$-Form and Global Relations}

\Author{Anthony C.L. ASHTON}

\AuthorNameForHeading{A.C.L. Ashton}

\Address{Department of Applied Mathematics and Theoretical Physics, \\
University of Cambridge, Cambridge, CB3 0WA, UK}

\Email{\href{mailto:a.c.l.ashton@damtp.cam.ac.uk}{a.c.l.ashton@damtp.cam.ac.uk}}

\ArticleDates{Received December 20, 2007, in f\/inal form March
03, 2008; Published online March 20, 2008}

\Abstract{In [{\it Proc. Roy. Soc. London Ser.~A} {\bf 453} (1997), no.~1962, 1411--1443]
A.S.~Fokas introduced a novel method for solving a large class of boundary value problems associated with evolution equations. This approach relies on the construction of a so-called global relation: an integral expression that couples initial and boundary data. The global relation can be found by constructing a dif\/ferential form dependent on some spectral parameter, that is closed on the condition that a given partial dif\/ferential equation is satisf\/ied. Such a~dif\/ferential form is said to be fundamental [{\it Quart. J. Mech. Appl. Math.} {\bf 55} (2002), 457--479].
We give an algorithmic approach in constructing a fundamental $k$-form associated with a given boundary value problem, and address issues of uniqueness. Also, we extend a result of Fokas and Zyskin to give an integral representation to the solution of a class of boundary value problems, in an arbitrary number of dimensions. We present an extended example using these results in which we construct a global relation for the linearised Navier--Stokes equations.}

\Classification{30E25; 35E99; 35P05}

\section{Introduction}
The traditional analysis associated with boundary value problems (BVPs) are the methods of classical transforms, of images and the use of fundamental solutions. These have been an incredibly successful set of tools and cover a large class of problems. They are not however, infallible -- an example would be the heat equation on the half line with moving boundary \cite{fokas2007gdn}. A~new, unif\/ied approach to studying BVPs was introduced by A.S.~Fokas in \cite{fokas1997utm} which allows the study of problems for which previous methods would fail. This method makes use of the classic integral theorems and elementary complex analysis -- using which it provides:
\begin{itemize}\itemsep=0pt
\item existence results;
\item well posedness conditions;
\item exact solutions in terms of Fourier type integrals.
\end{itemize}
The last of these points deserves elaboration. In the case of evolution equations on the half line
\[ \partial_tq(x,t) + \omega\left(-i\partial_x\right)q(x,t)=0, \qquad (x,t)\in\Omega = \mathbb{R}^+\times (0,T),\]
the new method provides an exact solution which takes the integral form
\[ \int_{\mathbb{R}} e^{ikx-\omega(k)t}\hat{q}_0 (k)\, \dd k - \int_{\partial D^+} e^{ikx-\omega(k)t} \tilde{g}(k)\, \dd k,\]
where $\hat{q}_0(k)$, $\tilde{g}(k)$ are known functions related to the initial and boundary data respectively. The contour $\partial D^+$ is determined by the polynomial $\omega(k) $, and lies in $\Im k\geq 0$. This solution is open to asymptotic analysis by elementary means, since all the $(x,t)$ dependence is in the exponential term. Also, the contours can be adjusted by use of Cauchy's theorem to give uniform convergence\footnote{Certain conditions must be placed on the function class to which $\hat{q}_0(k)$ belongs, but not overly restrictive.}. Numerical implementation of these solutions is then much simpler. For example, the classical solution to the heat equation on the half line is given by a Fourier series, whereas this method provides a representation in terms of uniformly convergent Fourier type integrals which are easily dealt with using standard numerical techniques.

This method of solving boundary value problems is reliant on the analysis of the \emph{global relation}. This is an integral relation, dependent on some spectral parameter, that couples all the data on the boundary. To construct this, one must f\/irst recast the governing PDE into a~divergence form, which is then open to applications of classical integral theorems.

We treat this construction in the setting of the exterior calculus, so by means of Stokes' theorem we may proceed to higher dimensions arbitrarily. Within this setting, we proceed formally: given a boundary value problem on $\Omega$ with associated PDE $\mathcal{L}q=0$, we construct a~dif\/ferential form $\eta = \eta(q,\sigma) \in \Lambda^k (\Omega)$ such that
\[ \dd \eta = 0 \quad \Leftrightarrow \quad\mathcal{L}q=0, \]
where $\sigma\in\mathbb{C}$ is a spectral parameter. Such a dif\/ferential form is said to be \emph{fundamental} \cite{fokas2002fdf}. We see that under the condition $\mathcal{L}q=0$ in $\Omega$, Stokes' theorem gives
\[ 0=\int_{\partial\Omega} \eta (q,\sigma), \]
which is the global relation. So for each global relation, there is a corresponding fundamental $k$-form and we use this as a basis for constructing global relations for systems of linear PDEs. We give a specif\/ic example by constructing the global relation for the unsteady Stokes' equations.

We then use the construction of the fundamental $k$-form to extend a result due to Fokas and Zyskin concerning integral representations of solutions to BVPs in two dimensions. This is achieved by proving their result in the context of the work in this paper, and then extending it to arbitrary dimension.

We conf\/ine our attention to linear partial dif\/ferential operators (PDOs) with constant coef\/f\/i\-cients, but the majority of the analysis can be carried over into non-constant coef\/f\/i\-cients routinely. More generally, one can work with a ring of pseudo-dif\/ferential operators (see \cite{dickey1990sea} for example), but we shall not pursue this line of work here.

\subsection{Organisation of the paper}

In Section~\ref{sec2} we give a formal outline of the construction of a fundamental $k$-form for a given BVP. This involves a series of elementary lemmas, the results of which can be applied algorithmically. We give an example of the construction for some simple dif\/ferential operators, and outline the use of the global relation in solving the classical problem associated with travelling waves on a string. A result from~\cite{fokas2002fdf} is then extended to give an integral representation of the solution (assuming existence) to a class of BVPs on $\Omega \subset \mathbb{R}^n$ in terms of the fundamental $k$-form. We carry over the analysis to deal with systems of PDEs in arbitrary dimension, and conclude in Subsection~\ref{sec2.4} with an extended example in which we construct the global relation for the unsteady Stokes equations in 3+1 dimensions. We aim to keep notation standard throughout, but the reader is referred to the appendix in case of confusion. There, we also include some of the longer calculations that are referred to in the main body of the paper.

\section[The fundamental $k$-form]{The fundamental $\boldsymbol{k}$-form}\label{sec2}

In this section we introduce the concept of the fundamental $k$-form: this is a dif\/ferential form associated with a given linear PDE such that the form is closed if the PDE is satisf\/ied. Use of a solution to the adjoint problem then allows us to enforce that the dif\/ferential form is closed if, \emph{and only if} the PDE is solved. This then gives an equivalent problem, involving integrals along the boundary of the domain of our PDE -- thus naturally incorporating the boundary data for the problem. We study these integral equations, or so-called global relations, and invoke spectral analysis to completely solve BVPs associated with systems of linear PDEs~\cite{fokas1997utm}.

\subsection{Formulation of the problem}

We def\/ine $\eta \in \Lambda^{n-1} \left(\Omega\right)$ for the linear PDO $\mathcal{L} = \sum c_\alpha \partial^\alpha$ (multi-index notation assumed) via
\begin{gather}
\mathrm d \eta = \big( \tilde{q}\mathcal{L} q - q\mathcal{L}^\dagger \tilde{q}\big) \mathrm d x^1 \wedge \cdots \wedge \mathrm d x^n , \label{eta}\end{gather}
where $\Omega \subset \mathbb{R}^n$ is a simply connected domain with piecewise smooth boundary $\partial\Omega$ and the $c_\alpha$ are constants. Later, we will introduce $\{q, \tilde{q}\}$ to be such that $\mathcal{L} q = \mathcal{L}^\dagger \tilde{q} = 0$ and which gives a~suf\/f\/icient condition for $\eta$ to be closed: that is $\mathrm d\eta = 0$. An application of Stokes' theorem then gives
\[ \int_{\partial\Omega}\eta (q,\tilde{q}) = 0. \]
It will be this equation that links the given data on the boundary of our domain, to the solution of our problem by introducing a system of spectral parameters $\{ \sigma_j\}$ such that $\tilde{q}=\exp \left( i\sigma_j x^j\right)$ satisf\/ies $\mathcal{L}^\dagger \tilde{q}=0$. These will allow us to invoke spectral analysis to solve the problem. It is possible to set this work up in full generality, using the theory of jet bundles to describe dif\/ferential operators on manifolds. However, this introduces technicalities that do not aid the presentation of the theory that follows -- so we shall not pursue it here. Essentially, the important property of the domain we have decided to work on is that it is cohomologically trivial: i.e.\ all closed forms are exact.

Since $\eta$ is closed, there $\exists \, \theta \in \Lambda^{n-2}(\Omega)$ such that $\eta = \mathrm d \theta$, which is simply a consequence of the Poincar\'e lemma and the fact $\Omega\subset\mathbb{R}^n$ is contractible. However, it seems easier to construct
\[ \eta = \textstyle \sum_j (-1)^{j+1}a_j (x) \mathrm dx^1 \wedge \cdots \wedge\widehat{ \mathrm dx^j} \wedge \cdots \wedge \mathrm dx^n, \]
where the hat indicates that the particular 1-form be removed. Indeed, this formulation leaves us to f\/ind the $\{ a_j (x)\}$ such that
\begin{gather}
\textstyle \sum_j \partial_j a_j (x) = \tilde{q}\mathcal{L} q - q\mathcal{L}^\dagger\tilde{q}, \label{div}
\end{gather}
where $\partial_j \equiv \partial / \partial x^j$. The $a_j(x)$ necessarily exist, assuming the existence of the adjoint, $\mathcal{L}^\dagger$. This follows from the def\/inition of the adjoint, since
\[ \langle \psi, \mathcal{L}\phi\rangle \defn \langle \mathcal{L}^\dagger \psi,\phi\rangle \]
for any test functions $\phi, \psi \in C^\infty_c (\Omega)$, and $\langle \cdot,\cdot\rangle$ is the inner product associated with $L^2(\Omega)$. Our task now is to construct the closed form $\eta$, given $\mathcal{L}$. Then, using a solution to the adjoint problem, which incorporates spectral parameters, the condition $\dd \eta = 0$ is then equivalent to $\mathcal{L}q=0$.
\begin{definition}\label{def1}
For $\alpha, \beta \in \mathbb{Z}^n_+$, the anti-symmetric object $[\alpha, \beta ]$ is def\/ined  via
\[ [\alpha, \beta ] \defn \partial^\beta \tilde{q}\, \partial^\alpha q - \partial^\beta q\, \partial^\alpha \tilde{q} \]
and the symmetric object $\{\alpha, \beta\}$ is def\/ined by
\[  \{\alpha, \beta \} \defn  \partial^\beta \tilde{q}\, \partial^\alpha q + \partial^\beta q\, \partial^\alpha \tilde{q}, \]
where the functions $q$, $\tilde{q}$ are suitably smooth.
\end{definition}

This notation will be used extensively in what follows: by proving a small collection of lemmas, we have means of decomposing the RHS of~\eqref{div} into a divergence form, hence f\/inding the unknown functions $a_j(x)$. The convenience of this notation is realised after observing the following result.
\begin{lemma}\label{lemma1}
Each PDO of the form $\mathcal{L}=\sum c_\alpha\partial^\alpha$ decomposes as
\[ \mathcal{L} = \mathcal{L}_{\mathrm{e}} \oplus \mathcal{L}_{\mathrm{o}}, \]
where $\mathcal{L}_{\mathrm{e}} = \mathcal{L}_{\mathrm{e}}^\dagger$ and $\mathcal{L}_{\mathrm{o}} = -\mathcal{L}_{\mathrm{o}}^\dagger$, i.e.\ they constitute the self-adjoint and skew-adjoint parts of the operator $\mathcal{L}$.
\end{lemma}

This lemma will allow us to deal with the symmetric and anti-symmetric terms that will appear in the sum separately, and simplify the analysis somewhat. The proof of the lemma is trivial: simply split the sum into parts in which $|\alpha|$ is even/odd. We are now left with the task of decomposing the sum into a divergence form, and using our new notation we have
\begin{gather} \big( \tilde{q}\mathcal{L} q - q\mathcal{L}^\dagger \tilde{q}\big)\equiv   \sum_{|\alpha|\, \mathrm{odd}} c_\alpha \{\alpha, 0\} + \sum_{|\alpha|\, \mathrm{even}} c_\alpha [\alpha, 0].
 \label{sum} \end{gather}
We now address the decomposition of these two terms separately.
\begin{lemma} \label{lem1}
For $1 \leq k \leq n$ we have the identity
\begin{gather*} [\alpha, \beta ] = \partial_k \big ( [ \alpha - e_k, \beta] - [\alpha - 2e_k, \beta + e_k ] + \cdots  \\
\phantom{[\alpha, \beta ] =}{}  + (-1)^{(\gamma_k +1)}[\alpha - \gamma_k e_k, \beta + (\gamma_k -1)e_k ]\big ) + (-1)^{\gamma_k}[\alpha - \gamma_ke_k, \beta + \gamma_k e_k ]
\end{gather*}
and the analogous result for $\{\cdot,\cdot\}$ is
\begin{gather*} \{\alpha, \beta \} = \partial_k \big ( \{ \alpha - e_k, \beta\} - \{\alpha - 2e_k, \beta + e_k \} + \cdots  \\
\phantom{\{\alpha, \beta \} =}{}  + (-1)^{(\gamma_k +1)}\{\alpha - \gamma_k e_k, \beta + (\gamma_k -1)e_k \}\big ) + (-1)^{\gamma_k}\{\alpha - \gamma_ke_k, \beta + \gamma_k e_k \}
\end{gather*}
for any $\gamma_k \leq \alpha_k$, where $\alpha = (\alpha_1, \ldots , \alpha_k, \ldots , \alpha_n )$ and the $\{e_j\}$ are the standard basis vectors on~$\mathbb{R}^n$.
\end{lemma}
\begin{proof}
We proceed by an elementary observation
\begin{gather*}
\partial^\beta \tilde{q} \, \partial^\alpha q = D^{\beta_1}_1 \cdots D^{\beta_k}_k\cdots  D^{\beta_n}_n \tilde{q} \, D^{\alpha_1}_1 \cdots D^{\alpha_k}_k\cdots D^{\alpha_n}_n q \\
\phantom{\partial^\beta \tilde{q} \, \partial^\alpha q}{} = D^1_k \big( D^{\beta_1}_1 \cdots D^{\beta_k}_k\cdots  D^{\beta_n}_n \tilde{q} \, D^{\alpha_1}_1 \cdots D^{\alpha_k - 1}_k\cdots D^{\alpha_n}_n q \big) \\
\phantom{\partial^\beta \tilde{q} \, \partial^\alpha q=}{}  - D^{\beta_1}_1 \cdots D^{\beta_k+ 1}_k\cdots  D^{\beta_n}_n \tilde{q} \, D^{\alpha_1}_1 \cdots D^{\alpha_k-1}_k\cdots D^{\alpha_n}_n q,
\end{gather*}
where $D^a_b \doteqdot \left( \partial /\partial x^b\right)^a$. Clearly then, we have
\[ [\alpha, \beta] = \partial_k [\alpha - e_k, \beta ] - [\alpha - e_k, \beta + e_k ]. \]
The result follows, inductively. The proof for $\{\cdot,\cdot\}$ is similar.
\end{proof}
\begin{lemma}\label{lem2}
For $\alpha \in\mathbb{Z}^n_+$ and $1\leq k\leq n$ we have
\[ \{\alpha, \alpha +e_k \} = \partial_k \{\alpha,\alpha\}, \]
where the $e_k$ are standard basis vectors on $\mathbb{R}^n$.
\end{lemma}
\begin{lemma}[Exchange lemma] \label{exchange}
For $\alpha \in \mathbb{Z}^n_+$ and $e_k$ the standard basis vectors in $\mathbb{R}^n$ we have
\begin{gather*}   \partial^{\alpha + e_1 + \cdots + e_m}q\, \partial^{\alpha + e_{m+1} + \cdots + e_{2m}}\tilde{q} = \partial^{\alpha + e_1 +\cdots+ \widehat{e_k}+\cdots + e_m+e_{m+j}}q \, \partial^{\alpha+e_{m+1} +\cdots+\widehat{e_{m+j}}+\cdots + e_{2m}+e_k}\tilde{q} \\
\phantom{\partial^{\alpha + e_1 + \cdots + e_m}q\, \partial^{\alpha + e_{m+1} + \cdots + e_{2m}}\tilde{q}}{} +\partial_k \big [ \partial^{\alpha + e_1 + \cdots + \widehat{e_k} + \cdots +e_{m}}q\, \partial^{\alpha + e_{m+1} +\cdots + e_{m+j} + \cdots +e_{2m}}\tilde{q}\big ] \\
\phantom{\partial^{\alpha + e_1 + \cdots + e_m}q\, \partial^{\alpha + e_{m+1} + \cdots + e_{2m}}\tilde{q}}{} - \partial_{m+j} \big [ \partial^{\alpha + e_1 + \cdots + \widehat{e_k} + \cdots +e_{m}}q\, \partial^{\alpha + e_{m+1} +\cdots +\widehat{e_{m+j}}+\cdots + e_{2m} + e_k} \tilde{q}\big ],
\end{gather*}
where $1\leq j,k \leq m$.
\end{lemma}

\begin{remark}In particular, this exchanges the partial derivatives between $q$ and $\tilde{q}$, modulo some total dif\/ferential, so repeated application gives
\[ [\alpha + e_1 +\cdots + e_m, \alpha + e_{m+1} + \cdots + e_{2m}] = \partial (\textrm{terms involving $q,\tilde{q},\partial q,\partial\tilde{q}$}) \]
since the f\/inal term, after exchanging each of the derivatives on the f\/irst term in $[\alpha + \cdots, \alpha + \cdots]$, cancels with the second term in $[\alpha + \cdots, \alpha + \cdots]$.
\end{remark}
The proof to these all follow from the def\/initions, so for the sake of brevity they are omitted. We now use these lemmas to give an algorithmic method of reducing the sum in (\ref{sum}) into a~divergence form. In what follows $\alpha_k$ will denote the $k$-th component of the vector $\alpha \in \mathbb{Z}^n_+$, and $|\alpha|=\alpha_1 + \cdots + \alpha_n$. It suf\/f\/ices to prove the result for a generic term involving $[\alpha,0]$ or $\{\alpha,0\}$, since the general result follows by linearity.

We observe that if each component of $\alpha$ is even, then $|\alpha|$ is even and we can use Lemma~\ref{lem1} on $[\alpha,0]$ exactly $\tfrac{1}{2}\alpha_k$ times for each $\alpha_k$, so that the f\/inal term will be $[\tfrac{1}{2}\alpha,\tfrac{1}{2}\alpha]$ which is zero by the anti-symmetry of $[\cdot,\cdot]$. Alternatively, we may have $|\alpha|$ even if all but some \emph{even} subset of the $\alpha_k$, $1\leq k\leq n$ are odd. Let this subset be of size $2m$, $1\leq m\leq \tfrac{1}{2}n$. We can make repeated use of lemma \ref{lem1} to decompose the object into something of the form
\[ [\gamma + e_1 + \dots + e_m, \gamma+ e_{m+1}+\dots+e_{2m}] + \partial (\textrm{other terms}), \]
where we have labelled $\alpha = 2\gamma + \sum\limits_{k=1}^{2m}e_k$ without loss of generality. This f\/inal term can be decomposed using the exchange Lemma~\ref{exchange}. This concludes the proof of the existence of the decomposition in the case for $|\alpha|$ is even. There is only one case to consider for $|\alpha|$ odd: that is, when there is an odd sized subset of the $\alpha_k$ that are themselves odd. Once again, using Lemma \ref{lem1} repeatedly, we may decompose our object into
\[ \{\gamma + e_1 + \cdots + e_m, \gamma + e_{m+1}+\cdots+e_{2m+1}\} + \partial (\textrm{other terms}) \]
and then apply the exchange lemma on the f\/irst term, $m$ times. Then, modulo total derivatives, we are left with the remainder from the exchange lemma, and the second term from $\{\alpha + \cdots, \alpha+\cdots\}$. These two terms can be written as $\{\beta, \beta + e_k\}$ for a suitable $\beta$ having chosen $e_k$, and by Lemma~\ref{lem2} we are done. This construction asserts the following theorem:
\begin{theorem}\label{theorem1}
Let $\mathcal{L}$ be a linear differential operator with constant coefficients. Given data on the (piece-wise smooth) boundary of a domain $\Omega\subset\mathbb{R}^n$, $\exists\, \eta \in \Lambda^{n-1}(\Omega)$ such that $\eta$ is closed iff $\tilde{q}\mathcal{L}q - q\mathcal{L}^\dagger \tilde{q} = 0$. Moreover, $\eta$ is local and bilinear in $q$, $\tilde{q}$ and derivatives thereof.
\end{theorem}

This essentially follows from the construction we outlined previously. Up to now, we have dealt in full generality -- in practice, performing the decomposition is a simple task of using Lemmas~\ref{lem1},~\ref{lem2},~\ref{exchange} repeatedly, there is very little computation involved. We proceed by way of an example to highlight the methods described thus far.

\begin{example}\label{ex1}
We consider the wave equation in 1+1 dimensions with $\mathcal{L}=\partial_t^2 - \partial_x^2$, on a f\/inite interval -- say $(x,t)\in \Omega = (0,l)\times (0,T)$ with initial and boundary conditions
\[ q(0,t)=q(l,t) =0, \qquad q(x,0)=u(x), \qquad q_t (x,0)=v(x). \]
Then employing the methods outlined previously, we f\/ind
\begin{gather*}
0 =  \tilde{q}\mathcal{L}q - q\mathcal{L}^\dagger \tilde{q}   = \tilde{q}q_{tt} - q\tilde{q}_{tt} -  \left( \tilde{q}q_{xx} - q\tilde{q}_{xx}\right) \\
\phantom{0}= \partial_t \left( \tilde{q}q_t - q\tilde{q}_t\right) - \partial_x \left(\tilde{q}q_x - q\tilde{q}_x\right),
\end{gather*}
from which we can introduce the closed 1-form $\eta \in \Lambda^1(\Omega)$, def\/ined by
\[ \eta = \left( \tilde{q}q_t - q\tilde{q}_t\right)\, \mathrm dx + \left(\tilde{q}q_x - q\tilde{q}_x\right)\, \mathrm dt . \]
Using two solutions for $\tilde{q} = \exp\{ ik(x\pm t)\}$, we f\/ind \emph{two} global relations by integrating about $\partial \Omega$ for the two dif\/ferent\footnote{See the Appendix for an outline of this computation and def\/initions of the notation used.} $\tilde{q}$
\begin{gather}
e^{-ikt}\left( \hat{q}_t (k,t) + ik\hat{q}(k,t)\right)  = \hat{g}(k) + ik\hat{f}(k) + e^{ikl}h_1(k,t) -  h_2 (k,t) , \label{global1}\\
e^{ikt}\left( \hat{q}_t (k,t) - ik\hat{q}(k,t)\right)  = \hat{g}(k) - ik\hat{f}(k) + e^{ikl}h_1(-k,t) -  h_2 (-k,t) , \label{global2}
\end{gather}
where we have invoked the boundary conditions at $x=0$ and $x=l$. These global relations are valid for all $k\in\mathbb{C}$, so in (\ref{global2}) we are free to take $k\mapsto -k$, and upon subtracting (\ref{global1}) we f\/ind
\[ e^{-ikt} \left( \hat{q}^s_t (k,t) + ik\hat{q}^s(k,t)\right) = \hat{g}^s(k) + ik\hat{f}^s(k) + \sin (kl)\, h_1 (k,t). \]
If we evaluate this expression for $kl = n\pi$, $n \in \mathbb{Z}$ then we eliminate the unknown data and f\/ind
\[ \partial_t \big\{ e^{ikt} \hat{q}^s(k,t) \big\} = e^{2ikt} \big\{ \hat{g}^s(k) + ik\hat{f}^s(k) \big\} , \qquad k = \frac{n\pi}{l}. \]
Integrating up, we f\/ind the classical Fourier series solution to the problem.
\end{example}
The previous example shows how this method works as expected for basic boundary value problems, but does not make much of the use of the main result so far; the next example will illustrate its use.
\begin{example}\label{ex2}
By means of a more involved example, we consider the dif\/ferential operator $\mathcal{L}$ def\/ined by
\[ \mathcal{L} = \partial_{xx}\partial_{yy}\partial_{zz} +  \partial_{xx}\partial_{yy}+ \partial_{zz}. \]
Equivalently, introducing the vectors $v_j \in \mathbb{Z}^3_+$ via
\[ v_1 = (2,2,2), \qquad v_2 = (2,2,0), \qquad v_3 = (0,0,2), \]
we may employ the same notation introduced earlier, so that
\[ \tilde{q}\mathcal{L}q - q\mathcal{L}^\dagger \tilde{q} = [v_1, 0] + [v_2, 0] + [v_3, 0]. \]
We now explicitly compute the f\/irst of these terms, with repeated use of the result in Lemma~\ref{lem1}. We have
\begin{gather*}
[v_1, 0]  = \partial_x \left [ (1,2,2),0 \right ] - \left [ (1,2,2), (1,0,0) \right ] \\
\phantom{[v_1, 0]}{} = \partial_x \left [ (1,2,2),0 \right ] - \partial_y \left [ (1,1,2), (1,0,0) \right ] + \left [ (1,1,2), (1,1,0)\right ] \\
\phantom{[v_1, 0]}{}= \partial_x \left [ (1,2,2),0 \right ] - \partial_y \left [ (1,1,2), (1,0,0) \right ] + \partial_z \left [ (1,1,1), (1,1,0) \right ],
\end{gather*}
where the f\/inal term $\left [ (1,1,1), (1,1,1)\right ]$ has been discarded, using the anti-symmetry of $[ \cdot, \cdot ]$. Making similar calculations for the last two terms, we f\/ind
\begin{gather*}
0= \partial_x \left\{ \tilde{q}q_{xyyzz} - q\tilde{q}_{xyyzz} + \tilde{q} q_{xyy} - q\tilde{q}_{xyy} \right\}
  - \partial_y \left\{ \tilde{q}_x q_{xyzz} - q_x \tilde{q}_{xyzz} + \tilde{q}_x q_{xy} - q_x \tilde{q}_{xy} \right\} \\
\phantom{0=}{} + \partial_z \left\{ \tilde{q}_{xy}q_{xyz} - q_{xy}\tilde{q}_{xyz} + \tilde{q}q_{z} - q\tilde{q}_{z} \right\}.
\end{gather*}
Now introducing a solution $\tilde{q}=\exp \left( i \sigma_j x^j\right)$ to $\mathcal{L}^\dagger \tilde{q}=0$, we f\/ind the constraint
\[ \sigma_1^2\sigma_2^2 = \frac{ \sigma_0^2}{1-\sigma_0^2}. \]
This constraint has a rational parameterisation: introduce $\lambda \in \mathbb{C}$, to f\/ind
\[ \sigma_1 \sigma_2 = \pm 2 \left( \lambda - \tfrac{1}{\lambda}\right)^{-1}, \qquad \sigma_0 = \pm 2\left( \lambda + \tfrac{1}{\lambda}\right)^{-1}. \]
This can then be used in our expression for $\eta$, from which we can again invoke Stokes' theorem to get an integral equation involving the spectral parameter $\lambda \in \mathbb{C}$, with some additional $\nu\in\mathbb{C}$ and using similar methods to the previous example (although this case will be considerably more laborious), we can solve the problem given suf\/f\/icient data on the boundary of our domain.
\end{example}

\begin{remark}
Note that the construction of $\eta$ in this proof does not provide a unique divergence form for the dif\/ferential operator -- the construction is dependent on the order in which we apply the result of Lemma~\ref{lem1} to each component of $\alpha$, as well as how we apply Lemma~\ref{exchange}.
\end{remark}

Clearly, the lack of uniqueness means there is a certain amount of ambiguity in how one should construct $\eta$ for a given PDE. In general, it should be chosen so that the resulting integral equation contains terms which are specif\/ied as much as possible by the given data on the boundary. We illustrate this fact by counting the number of dif\/ferent ways to decompose a~single term in~(\ref{sum}). Without loss of generality we set $\alpha = 2\gamma + \sum\limits_{i=1}^m e_i$ for some $0\leq m \leq n$.
\begin{definition}
Given $\alpha \in \mathbb{Z}^n_+$, we def\/ine a path, $\mathcal{P}(\alpha)$, to be an ordered collection of basis vectors such that $[\alpha,0]$ (alternatively $\{\alpha,0\}$) is decomposed into $[\gamma + \sum_i e_i, \gamma]$ (resp. $\{\gamma+\sum_i e_i,\gamma\}$) by applying Lemma~\ref{lem1} to the terms in $2\gamma$ in the order $\mathcal{P}(\alpha)$.
\end{definition}
\begin{example}
Let $\alpha = (2,2,4) \equiv 2(1,1,2)$, so that $[\alpha,0]$ can be decomposed by applying Lemma~\ref{lem1} to the terms $\{e_1, e_2, e_3, e_3\}$ in some order. For instance, applying it along the path $\mathcal{P}_1(\alpha)=\{ e_1, e_2, e_3,e_3\}$ yields
\begin{gather*}
[\alpha,0] = \partial_1 [e_1 + 2e_2 + 4e_3,0] - \partial_2 [ e_1 + e_2 + 4e_3, e_1] + \partial_3 [ e_1 + e_2 + 3e_3, e_1+e_2] \\
\phantom{[\alpha,0] =}{} - \partial_3[e_1 + e_2 + 2e_3, e_1 + e_2 + e_3].
\end{gather*}
Alternatively, if we apply the lemma along the path $\mathcal{P}_2(\alpha) = \{ e_3, e_3, e_1, e_2\}$ we f\/ind
\begin{gather*}
 [\alpha,0] = \partial_3 [2e_1 + 2e_2 + 3e_3,0] - \partial_3[2e_1 + 2e_2 + 2e_3, e_3] + \partial_1[e_1 + 2e_2 + 2e_3, 2e_3] \\
 \phantom{[\alpha,0] =}{} - \partial_2[e_1 + e_2 + 2e_3, 2e_3 + e_1].
\end{gather*}
\end{example}

This example illustrates how the number of dif\/ferent paths for a term in the PDO will ef\/fect the number of dif\/ferent fundamental $k$-forms one may construct by means of our algorithm. Indeed, we see that Lemma~\ref{lem1} can be applied in a variety of ways to a single term, the total number of which would be equal to the total number of distinct permutations of a generic path. For instance, in the previous example we could have constructed a total of $4!/2!=12$ distinct paths, which is exactly the number of proper permutations of the set $\{ e_1, e_2, e_3, e_3\}$.
\begin{definition}
Given $\alpha \in \mathbb{Z}^n_+$ and a path $\mathcal{P}(\alpha)$ we denote the number of distinct permutations of the elements of the path $\mathcal{P}(\alpha)$ by $\sigma(\alpha)$.
\end{definition}

\begin{example}
Given $\alpha = (2,2,5,6)\equiv 2(1,1,2,3)+(0,0,1,0)$, a path must consist of the elements $\{ e_1, e_2, e_3, e_3, e_4, e_4, e_4\}$. It follows that $\sigma(\alpha) = 7!/(2!\times 3!) = 420$.
\end{example}

\begin{remark}
Note that the length of a path will be equal to $\tfrac{1}{2} \left( |\alpha| - \# \text{ of odd components in } \alpha\right)$.
\end{remark}

Of course, there is more to be done if some of the $\alpha_k$, $1\leq k\leq n$ happen to be odd: we need to apply the exchange lemma a def\/inite number of times. We know that we can apply Lemma~\ref{lem1} successively to reduce a generic term into either
\begin{gather} [\gamma + e_1 + \cdots + e_{2m}, \gamma ]\label{even} \end{gather}
in the case of an even number of odd components of $\alpha$, or into
\begin{gather} \{ \gamma + e_1 + \cdots + e_{2m+1}, \gamma \} \label{odd}\end{gather}
in the case of an odd number of odd components of $\alpha$. Concentrating on (\ref{even}) f\/irst, we see there are $\binom{2m}{m}$ ways to manipulate this object, by means of Lemma~\ref{lem1}, into the form
 \[ [\gamma + e_{k_1} + \cdots + e_{k_m}, \gamma + e_{k_{m+1}} + \cdots + e_{k_{2m}} ], \]
where the $\{e_{k_i}\}$ are some permutation of the $\{e_i\}$. Now in this form, we have $m\times m$ choices for the f\/irst application of the exchange lemma, then $(m-1)\times (m-1)$ for the next, etc. So the total number of ways of applying the exchange lemma to this term is $(m!)^2$. And so the total number of ways of decomposing (\ref{even}) into divergence form is
\[ \frac{(2m)!}{m!\times m!} \times m! \times m! = (2m)!.  \]
Similarly, there are $\binom{2m+1}{m}$ ways to manipulate (\ref{odd}) into the form
\[ \{ \gamma + e_{k_1} + \cdots + e_{k_m}, \gamma + e_{k_{m+1}} + \cdots + e_{k_{2m+1}} \} \]
by means of Lemma~\ref{lem1}. We then have $(m+1)\times m$ choices for the f\/irst application of the exchange lemma, followed by $m\times (m-1)$  choices etc. It follows that the total number of ways of decomposing (\ref{odd}) into a total divergence is
\[ \frac{ (2m+1)!}{(m+1)!\times  m!} \times (m+1)! \times m! = (2m+1)!. \]
In both cases then, the total number of dif\/ferent decompositions of the terms (\ref{even}) and (\ref{odd}) is simply $(\# \text{ of odd terms in $\alpha$})!$. We can now state the total number of fundamental $k$-forms that can be constructed using out method.
\begin{theorem}\label{theorem2}
Given a linear PDO with constant coefficients $\mathcal{L}=\sum_\alpha c_\alpha \partial^\alpha$, our method can construct $N(\mathcal{L})$ different fundamental $k$-forms, where
\[ N(\mathcal{L}) = \prod_\alpha O_\alpha !\, \sigma(\alpha) \]
and $O_\alpha$ denotes the number of odd components of $\alpha \in \mathbb{Z}^n_+$.
\end{theorem}
The proof follows from the arguments previous based on a single term. The total number is found by simply taking the product over all the terms in the PDO. From this result it immediately follows that in Example~\ref{ex1}, the $k$-form we constructed was one of
\[ N(\mathcal{L}) = 1 \times 1 = 1 \]
for $\mathcal{L}=\partial_{tt}-\partial_{xx}$. So in this case, it  was in fact the unique fundamental form for the BVP. However, for the higher order operator in Example~\ref{ex2}, where $\mathcal{L}= \partial_{xx}\partial_{yy}\partial_{zz} +  \partial_{xx}\partial_{yy}+ \partial_{zz}$, the total number of distinct $k$-forms we could construct is
\[ N(\mathcal{L}) = 3! \times 2! \times 1! = 12, \]
any of which could be used to solve the BVP, given suf\/f\/icient data on $\partial \Omega$.
\begin{example}
We consider bi-harmonic functions on $\mathbb{R}^3$, i.e.\ classical solutions to $\triangle ^2 q = 0$. In Cartesian coordinates then, the relevant dif\/ferential operator is given by
\[ \mathcal{L} = \partial_{xxxx} + \partial_{yyyy} + \partial_{zzzz} + 2\partial_{xxyy} + 2\partial_{yyzz} + 2\partial_{zzxx}. \]
Then employing the notation introduced earlier, we seek to decompose
\[ [v_1, 0]+ [v_2, 0]+ [v_3, 0]+2 [v_4,0]+2 [v_5,0]+2 [v_6,0], \]
where the $\{v_j\}$ are def\/ined by
\begin{gather*} v_1 = (4,0,0),  \qquad v_4 = (2,2,0), \\
v_2 = (0,4,0),  \qquad v_5 = (0,2,2), \\
v_3 = (0,0,4),  \qquad v_6 = (2,0,2). \end{gather*}

From which, we can read of\/f the number of distinct fundamental $k$-forms produced by our method to be
\[ N(\mathcal{L}) = 1!\times1!\times1!\times2!\times2!\times2! = 8. \]
Below is an example of one such decomposition that can be used to construct the fundamental form $\eta$, where the $\{\sigma_j\}$ are def\/ined so that $\sum_j \sigma_j^2 =0$ (so that the adjoint solution $\tilde{q}=\exp (i\sigma_j x^j )$ solves $\triangle ^2 \tilde{q}=0$)
\begin{gather*} 0 =\partial_x \left(  e^{i \sigma_jx^j} \left\{ q_{xxx}+ i\sigma_1^3 q -i\sigma_1 q_{xx} - \sigma_1^2 q_x + 2q_{xyy}+i\sigma_1\sigma_2^2 q + 2q_{xzz} + 2i\sigma_1 \sigma_3^2 q \right\}\right) \\
 \phantom{0=}{} + \partial_y\left(  e^{i \sigma_jx^j} \{  q_{yyy} + i\sigma_2^3 q - i\sigma_2 q_{yy}-\sigma_2^2 q_y - 2i\sigma_1 q_{xy} - 2\sigma_1\sigma_2 q_x + 2q_{yzz} + 2i\sigma_1\sigma_3^2 q \} \right) \\
 \phantom{0=}{}  +\partial_z \left(  e^{i \sigma_jx^j}\{q_{zzz} +i\sigma_3^3 q -i\sigma_3 q_{zz} - \sigma_3^2 q_z -2i\sigma_1 q_{xz} -2 \sigma_1\sigma_3 q_x - 2i\sigma_2 q_{yz} - 2\sigma_2\sigma_3 q_y \}\right),
 \end{gather*}
where $(x,y,z)\equiv (x^1, x^2, x^3)$. In constructing this equation, we simply made repeated use of Lemma~\ref{lem1} -- there is very little computation involved.
\end{example}
We must now address the question of whether $N(\mathcal{L})$ is well def\/ined -- that is to say: is $N(\mathcal{L})$ a~coordinate independent object? Suppose $\Sigma_1$ is the set of fundamental $k$-forms associated with~$\mathcal{L}$ on the domain $\Omega_1$, and let $\rho:\Omega_1 \rightarrow \Omega_2$ be a dif\/feomorphism, with inverse $\phi$. Then the question becomes: is $N(\mathcal{L})$ on $\Omega_1$ as on $\Omega_2$. This is partially  addressed in the following lemma.

\begin{lemma}\label{lemma5}
Each fundamental $k$-form in $\Omega_1$ induces a corresponding fundamental $k$-form in $\Omega_2$ and vice versa.
\end{lemma}

\begin{proof}
Take $\eta \in \Sigma_2$, i.e.\ $\eta$ is fundamental on $\Omega_2$. Then the pull-back $\rho^*$ induces a fundamental form on $\Omega_1$. Indeed, setting $\theta = \rho^* \eta$, we f\/ind
\[ \dd \theta = \dd \left(\rho^* \eta\right) = \rho^* \dd \eta = 0 \]
since the pull-back commutes with the exterior derivative. We conclude that for each element of $\Sigma_2$ there is a corresponding one in $\Sigma_1$. Using the same argument with the pull-back $\phi^*$, completes the proof.
\end{proof}

The reason this doesn't fully answer our question, is because of the following. Let $\eta$, $\theta$ be two fundamental $k$-forms on $\Omega$ constructed via our method and consider $\phi^*\eta$ and $\phi^*\theta$. These are fundamental on $\rho(\Omega)$, but \emph{are not} necessarily of the form constructed via our method. That is to say, applying our method to $\mathcal{L}$ \emph{after} applying the dif\/feomorphism to $\Omega$, wont necessarily yield the same set of fundamental $k$-forms as it would if we computed the induced fundamental $k$-forms via the appropriate pull-back.

Now although we have shown that many fundamental forms may be constructed, they are not necessarily independent. We use the fact that there is no cohomology on the class of domains we are working on, so any closed form is exact. Now consider two fundamental $k$-forms $\eta$, $\theta$ that dif\/fer by some closed (hence exact) form $\dd \psi$. These two forms yield the same global relation, indeed
\[ \int_{\partial\Omega} (\eta-\theta) = \int_{\partial\Omega} \dd\psi = 0 \]
since $\dd^2 = 0$. We should conf\/ine our attention then, to the equivalence class of fundamental forms in which $\eta_1 \sim \eta_2\, \Leftrightarrow\, \eta_1 - \eta_2 = \dd \theta$, for some $\theta \in \Lambda^{k-1}(\Omega)$. Labelling the set of exact dif\/ferential forms on $\Omega$ by $\Pi$, and the set of fundamental $k$-forms on $\Omega$ by $\Sigma$, we have the following result.
\begin{lemma}
There is a unique fundamental $k$-form in $\Sigma / \Pi$.
\end{lemma}
\begin{proof}
Assume that $\Sigma/\Pi$ contains $m>1$ fundamental $k$-forms, and we consider the dif\/ference of two such $k$-forms: $\psi = \eta_i - \eta_j$, $i\neq j$. Note that $\mathrm d\psi = 0$ \emph{regardless} of whether $\tilde{q}\mathcal{L}q=q\mathcal{L}^\dagger \tilde{q}$. Since we are working on a cohomologically trivial space, it follows that $\exists\, \omega \in \Lambda^{k-1}(\Omega)$ such that $\psi = \mathrm d\omega$. However, this implies that $\eta_i$ and $\eta_j$ are equivalent in $\Sigma /\Pi$, contradicting our hypothesis. This is true for each pair $(i,j)$, so we may conclude that $\Sigma / \Pi$ contains one unique element.
\end{proof}

The consequences of this fact are as follows. Even though we have shown that we can construct several dif\/ferent fundamental $k$-forms for a given PDE, they are all equivalent up to some closed form, the closeness of which \emph{does not} depend on the condition $\tilde{q}\mathcal{L}q =q\mathcal{L}^\dagger \tilde{q}$. Therefore, any of the fundamental $k$-forms we construct will generate the same global relation. In practice then, we should choose a particular decomposition that yields a fundamental $k$-form whose coef\/f\/icients contain derivatives of $q$ that match as closely as possible, the prescribed data on the boundary. This will then lead to a global relation that can be dealt with ef\/f\/iciently, by using the appropriate data on the boundary.

\subsection{Integral representations}
Here we use the fundamental $k$-form to give an explicit solution to a BVP associated with the polynomial dif\/ferential operator $\mathcal{L} = \mathcal{L}(\partial_1, \partial_2, \ldots, \partial_n)$. We expand on the observation given below:
\begin{lemma}[Fokas \& Zyskin]
Given a polynomial differential operator $\mathfrak{p}(\partial_1, \partial_2)$, and $\eta \in \Lambda^1(\Omega)$ such that
\[ \dd \eta = \mathfrak{p}(ik_1, ik_2)q(x_1, x_2) e^{-ik_1x_1-ik_2x_2}\, \dd x_1 \wedge \dd x_2, \]
where $\mathfrak{p}(\partial_1, \partial_2)q=0$ in $\Omega$, then if $q$ exists it can be represented in the form
\[ q(x_1, x_2) = \frac{1}{(2\pi)^2} \int_{\mathbb{R}^2}\dd k_1\, \dd k_2 \int_{\partial\Omega} \frac{e^{ik_1x_1 + ik_2x_2}\eta (y_1, y_2, k_1, k_2)}{\mathfrak{p}(ik_1, ik_2)}, \qquad (y_1, y_2) \in \partial\Omega. \]
\end{lemma}

A proof of this is found in \cite{fokas2002fdf}, it essentially follows from the distributional relation
\[ \int_{\mathbb{R}} \dd k\, e^{ik(x-y)} = 2\pi \delta (x-y). \]
We can now extend this result to arbitrary dimension as follows. We construct an fundamental $k$-form $\eta$ using the lemmas of the previous section, but we \emph{do not} solve the adjoint problem. Instead, we introduce $\tilde{q} = \exp(-ik_ix^i)$ from which it follows that
\[ \mathcal{L}^\dagger \tilde{q} = \mathcal{L}(-\partial_1, -\partial_2, \ldots, -\partial_n) e^{-ik_ix^i} = \mathcal{L}(ik_1, ik_2, \ldots, ik_n) e^{-ik_ix^i}.\]
Now using equation (\ref{eta}), we f\/ind that
\[ \dd \eta = - \mathcal{L}(ik_1, ik_2, \ldots, ik_n) q(x_1, \ldots, x_n) e^{-ik_ix^i}\, \dd x^1 \wedge \cdots \wedge \dd x^n \]
assuming $\mathcal{L}q=0$ in $\Omega$. This then gives the integral representation for $q$ as
\[ q(x) = \frac{-1}{(2\pi)^n}\int_{\mathbb{R}^n} \dd k_1\cdots \dd k_n \int_{\partial\Omega} \frac{e^{ik_i x^i} \eta (y,k)}{\mathcal{L}(ik_1, \ldots,ik_n)}, \qquad y\in\partial\Omega. \]
We summarise this result in the following.
\begin{theorem}
Given a boundary value problem in $\Omega$, with associated polynomial differential operator $\mathcal{L}=\mathcal{L}(\partial_1, \ldots, \partial_n)$, on the assumption that a solution exists, it can be represented by
\[ q(x) = \frac{-1}{(2\pi)^n}\int_{\mathbb{R}^n} \dd k_1\cdots \dd k_n \int_{\partial\Omega} \frac{e^{ik_i x^i} \eta (y,k)}{\mathcal{L}(ik_1, \ldots,ik_n)},  \]
where $y\in\partial\Omega$ and $\eta$ is a fundamental $k$-form for $\mathcal{L}$ associated with the adjoint solution $e^{-ik_ix^i}$.
\end{theorem}

\begin{remark}
Note that this representation is in accordance with the Ehrenpreis principle from the theory of complex analysis of several variables. See for instance \cite{ehrenpreis1970fas}.
\end{remark}

The term `adjoint solution' is a misnomer, $\tilde{q}$ does \emph{not} solve the adjoint problem. It is simply the function we choose for $\tilde{q}$ in the decomposition of (\ref{sum}). Note that the non-uniqueness of $\eta$ discussed earlier, does not af\/fect this result. Indeed, the addition of some exact form to $\eta(x,k)$ will not contribute since $\int_{\partial\Omega}\dd \psi (x,k) = 0$.

\subsection{Systems of PDEs}
Here we extend the results proven in the last section, to deal with systems of coupled PDEs. The vector $\phi = (\phi_1, \phi_2, \dots, \phi_n)$ will be denoted by $\phi_i$ and we shall assume summation convention throughout. The dif\/ferential operator $\mathcal{L}$ will be expressed component wise, so that we denote $\mathcal{L}=\mathcal{L}_{ij}$. Studying linear systems of the form $\mathcal{L}[\phi]=0$ then translated to solving the coupled, linear set of PDEs
\[ \begin{bmatrix} \mathcal{L}_{11} & \cdots & \mathcal{L}_{1n} \\
\vdots & \ddots & \vdots \\
\mathcal{L}_{n1} & \cdots & \mathcal{L}_{nn} \end{bmatrix} \begin{bmatrix} \phi_1 \\
\vdots \\
\phi_n \end{bmatrix} = 0. \]

From this, we construct the formal adjoint, $\mathcal{L}^\dagger$ def\/ined by
\[\mathcal{L}^\dagger = \begin{bmatrix} \mathcal{L}_{11}^\dagger & \cdots & \mathcal{L}_{n1}^\dagger \\
\vdots & \ddots & \vdots \\
\mathcal{L}_{1n}^\dagger & \cdots & \mathcal{L}_{nn}^\dagger \end{bmatrix}.\]
We then decompose an analogous object to that in the previous section, given by
\[ \tilde{\phi}_i \mathcal{L}_{ij}\phi_j - \phi_i \{ \mathcal{L}^\dagger\}_{ij} \tilde{\phi}_j.  \]
Using the def\/inition of the adjoint, and changing the order of summation in the second term, we see this is equivalent to decomposing
\[ \tilde{\phi}_i \mathcal{L}_{ij}\phi_j - \phi_j  \mathcal{L}^\dagger_{ij} \tilde{\phi}_i. \]
Now we may apply the results of the previous section to each of the $n^2$ terms in this sum. Since we will want the closeness of the fundamental $k$-form to be equivalent to $\mathcal{L}[\phi]=0$, we must choose the solution to the adjoint problem so that each of the scalar f\/ields $\tilde{\phi}_i$ are linearly independent functions. With this condition, we see that we must have $\mathcal{L}_{ij}\phi_j=0$, $i=1, \ldots , n$, for the fundamental $k$-form to be closed. This is exactly the condition $\mathcal{L}[\phi]=0$.

\subsection{Unsteady Stokes equations}\label{sec2.4}

In this section we construct the fundamental 3-form associated with the unsteady, unforced Stokes equations, which are a direct linearisation of the classical Navier--Stokes equations. We use a separable to solution to the adjoint problem suitable for geometries which lend themselves to Cartensians. We deal with a divergence free velocity f\/ield $u(\mathbf{x},t)=(u_1, u_2, u_3)$ with and some scalar f\/ield $p(\mathbf{x},t)$, such that
\begin{gather}
\partial_t u_1 - \nu\Delta u_1 + \partial_x p = 0, \label{stokes1}\\
\partial_t u_2 - \nu\Delta u_2 + \partial_y p = 0, \\
\partial_t u_3 - \nu\Delta u_3 + \partial_z p = 0, \\
\partial_x u_1 + \partial_y u_2 +\partial_z u_3 =0, \label{stokes4}
\end{gather}
where $\nu >0$ is constant and $\Delta$ is the Laplacian on $\mathbb{R}^3$. We now introduce the 4-vector $\phi (\mathbf{x},t) = (\phi_1, \phi_2, \phi_3,\phi_4)$ where $\phi_4 = p(\mathbf{x},t)$ and $\phi_i (\mathbf{x},t) = u_i (\mathbf{x},t)$ for $i=1,2,3$. With this notation, equations (\ref{stokes1})--(\ref{stokes4}) take on the form $\mathcal{L}\phi = 0$, where
\[ \mathcal{L} \defn \begin{bmatrix} \partial_t - \nu \Delta & 0 & 0 & \partial_x \\ 0 & \partial_t - \nu \Delta & 0 & \partial_y \\ 0&0&\partial_t - \nu \Delta&\partial_z \\ \partial_x & \partial_y & \partial_z & 0 \end{bmatrix}. \]
We now employ the methods introduced in the last section to analyse this coupled system of linear PDEs, and see how the fundamental 3-form associated with the problem can give insights into the associated IBVP. Note f\/irst
\[ \mathcal{L}^\dagger = \begin{bmatrix} -\partial_t - \nu \Delta & 0 & 0 & -\partial_x \\ 0 & -\partial_t - \nu \Delta & 0 & -\partial_y \\ 0&0&-\partial_t - \nu \Delta&-\partial_z \\ -\partial_x & -\partial_y & -\partial_z & 0 \end{bmatrix}.\]
We f\/irst solve the associated adjoint problem, def\/ined by $\mathcal{L}^\dagger\tilde{\phi}=0$, incorporating a spectral parameter. In doing this, we are able to integrate this closed form over our domain $\Omega$, from which Stokes' theorem allows us to realise a constraining equation involving Fourier-type integrals of data on $\partial\Omega$.
\begin{lemma}
A solution to the adjoint problem of the required form is given by
\[ \tilde{\phi} = \begin{bmatrix} \mathbf{k} \\ \xi_3 \end{bmatrix}e^{-i\mathbf{k}\cdot\mathbf{x} +i \xi_3 t}, \]
where $\mathbf{k} = \mathbf{k}(\xi)$ is an isotropic vector in $\mathbb{C}^3$, $\xi = (\xi_1, \xi_2)$ is a $2$-spinor and $\xi_3 \in \mathbb{C}$.
\end{lemma}

This result may be checked routinely. The appearance of the spinor has an important consequence, as we shall now see.
\begin{lemma}
There are at least 2 global relations associated with the IBVP for the unsteady Stokes equations.
\end{lemma}
\begin{proof}
This is essentially a consequence of the fact $\mathrm{Spin}(3)\cong SU(2)$ is (def\/ined as) a double cover of $SO(3)$. Rotating $\mathbf{k}$ by $2\pi$ about some axis must leave it invariant so still provide a~solution to the adjoint problem, but the associated spinor transforms as $\xi \mapsto U\xi$ for $ U \in SU(2)$, with $U\neq \mathbb{I}$.
\end{proof}
We now decompose (\ref{stokes1})--(\ref{stokes4}) into a divergence form, from which we may construct the associated fundamental form. We omit the details of this, but it follows from Lemmas~\ref{lem1}, \ref{lem2}, \ref{exchange} that the unsteady stokes equations imply
\[ \frac{\partial \rho}{\partial t} + \text{Div}\,  J = 0, \]
where `density' $\rho$ and the components of the associated `f\/lux' $J = (J^1, J^2, J^3)$, are given by
\begin{gather*}
\rho = \tilde{\phi}_1 u_1 + \tilde{\phi}_2 u_2 + \tilde{\phi}_3 u_3, \\
J^1 =\tilde{\phi}_1 p + \phi_1 \tilde{p}+  \nu \big( u_1 \partial_x \tilde{\phi}_1 - \tilde{\phi}_1\partial_x u_1 + u_2 \partial_x \tilde{\phi}_2 - \tilde{\phi}_2\partial_x u_2+u_3 \partial_x \tilde{\phi}_3 - \tilde{\phi}_3\partial_x u_3 \big), \\
J^2 = \tilde{\phi}_2 p +\phi_2 \tilde{p}+  \nu \big( u_1 \partial_y \tilde{\phi}_1 - \tilde{\phi}_1\partial_y u_1 + u_2 \partial_y \tilde{\phi}_2 - \tilde{\phi}_2\partial_y u_2+u_3 \partial_y \tilde{\phi}_3 - \tilde{\phi}_3\partial_y u_3 \big), \\
J^3 = \tilde{\phi}_3 p +\phi_3 \tilde{p}+  \nu \big( u_1 \partial_z \tilde{\phi}_1 - \tilde{\phi}_1\partial_z u_1 + u_2 \partial_z \tilde{\phi}_2 - \tilde{\phi}_2\partial_z u_2+u_3 \partial_z \tilde{\phi}_3 - \tilde{\phi}_3\partial_x u_3 \big),
\end{gather*}
where $p=p(x,t)$, $u_i = u_i (x,t)$ etc. We now construct the fundamental 3-form associated with the problem, in which we encompass our previous solution to the adjoint problem. That is, we conclude that if $\mathcal{L}\phi=0$ then the following dif\/ferential form is closed
\[ \eta = \rho\, \dd x\wedge \dd y\wedge \dd z -J^1\dd y\wedge \dd z\wedge \dd t + J^2 \dd x\wedge \dd z\wedge \dd t - J^3 \dd x\wedge \dd y\wedge \dd t. \]
And so is \emph{fundamental}. We note that the terms in the above dif\/ferential form contain the $(\xi_1, \xi_2, \xi_3)$, which serve as spectral parameters. The global relation is then formed by integrating this form over $\Omega$, using data prescribed on the boundary. This results in an equation involving Fourier type integrals, which lend themselves to analysis as seen in \cite{fokas2002fdf}.

\appendix

\pdfbookmark[1]{Appendix A: Notation}{appendixA}
\section*{Appendix A: Notation}
We will use $\Omega$ to denote a simply connected subset of $\mathbb{R}^n$ with piecewise smooth boundary $\partial\Omega$. $L^2(\Omega)$ will denote the Hilbert space of square integrable functions w.r.t.\ the Lebesgue measure on $\Omega $, $\Lambda^k (\Omega)$ denotes the vector space of dif\/ferential $k$-forms on $\Omega$, $\dd :\Lambda^k (\Omega) \rightarrow \Lambda^{k+1} (\Omega)$ is the exterior derivative. When working with a fundamental $k$-form, $k=n-1$, i.e.\ one less than the dimension of the domain we are working on. We assume Stokes' theorem: on a $n$-manifold $\mathcal{M}$ with boundary $\partial\mathcal{M}$, we have
\[ \int_{\mathcal{M}} \dd \eta = \int_{\partial\mathcal{M}} \eta \]
for $\eta \in \Lambda^{n-1}(\mathcal{M})$ with compact support. The space of test functions on $\Omega$ will be denoted~$C^\infty_c (\Omega)$, meaning the space of smooth functions on $\Omega$ with compact support. $\mathcal{L}$ will denote a~li\-near dif\/ferential operator, with associated \emph{formal} adjoint $\mathcal{L}^\dagger$ with respect to the inner product on~$L^2(\Omega)$. The Greek letters $\alpha, \beta \ldots \in \mathbb{Z}^n_+$ will be used in mutli-index notation. The letters $i,j,k$ will denote numbers ranging from $1$ to $n$. Summation convention will be used.

\pdfbookmark[1]{Appendix B: Wave equation in $1+1$ dimensions}{appendixB}
\section*{Appendix B: Wave equation in $\boldsymbol{1+1}$ dimensions}

In Example~\ref{ex1} we construct the 1-form $\eta \in \Lambda^1 (\Omega)$, given by
\[ \eta = \left( \tilde{q}q_t - q\tilde{q}_t\right)\, \mathrm dx + \left(\tilde{q}q_x - q\tilde{q}_x\right)\, \mathrm dt.\]
Letting $\tilde{q}$ solve the adjoint problem, so $\tilde{q}=\exp[ ik(x- t)]$, enforces the closeness of $\eta$ is equivalent to $q$ solving $\square q=0$. We integrate this dif\/ferential form on $\partial\Omega$, and it follows from Stokes' theorem that
\[ \int_{\partial\Omega} e^{ik(x-t)} \left[ \left(q_t + ikq\right)\, \dd x + \left(q_x -ikq\right)\, \dd t \right] = 0 . \]
Now we carry out the integration, using the following notation for the initial data
\[ \hat{q}(k,t) = \int_0^l e^{ikx}q(x,t)\, \dd x, \qquad \hat{g}(k) = \int_0^l e^{ikx}q(x,0)\, \dd x, \qquad \hat{f}(k) = \int_0^l e^{ikx}q_t(x,0)\, \dd x \]
and the following notation for the relevant terms on the boundary are
\[ h_1 (\omega,t) = \int_0^t e^{-i\omega\tau}q_x(l,t)\,\dd \tau, \qquad h_2 (\omega,t) = \int_0^t e^{-i\omega\tau}q_x(0,t)\,\dd \tau. \]
We also introduce the notation $\hat{f}^s(k)$ to be def\/ined by
\[\hat{f}^s(k) = \int_0^l \sin (kx) f(x)\, \dd x. \]
With these def\/initions, two global relations in Example~\ref{ex1} follow.

\subsection*{Acknowledgements}
The author thanks Athanasios Fokas for introducing him to this subject area. Also thanks for Chris Taylor and Euan Spence for many helpful discussions. This work is supported by an EPSRC studentship.

\pdfbookmark[1]{References}{ref}
\LastPageEnding

\end{document}